\documentclass[12pt]{amsart}
\usepackage[margin=3cm]{geometry}   
\usepackage[T1]{fontenc}
\usepackage[latin1]{inputenc}
\usepackage{amssymb}
\usepackage{latexsym}
\usepackage{amsmath}

\usepackage{graphicx}
\input xy
\input diagxy
\xyoption{all}

\newcommand{\omitthis}[1]{}

\newcommand{\longtext}[1]{}
\newcommand{\shorttext}[1]{}
\newcommand{\commentaway}[1]{}

\newtheorem{thm}{Theorem}[section]
\newtheorem{cor}[thm]{Corollary}

\title{On equality of objects in categories in constructive type theory}

\author{Erik Palmgren}
\address{Department of Mathematics, Stockholm University, SE--106 91 Stockholm, Sweden}
\email{palmgren@math.su.se}
\urladdr{www.math.su.se/$\sim$palmgren}

\date{August 5, 2017}

\begin{document}

\begin{abstract} 
In this note we remark on the problem of equality of objects in categories formalized in Martin-L\"of's
constructive type theory.  A standard notion of category in this system
is E-category, where no such equality is specified. The main observation here is
that there is no general extension of E-categories to categories with equality on objects, unless
the principle Uniqueness of Identity Proofs (UIP) holds.  We also introduce 
the notion of an H-category, a variant of category with equality on objects, 
which makes it easy to compare to the notion of univalent category proposed 
for Univalent Type Theory by Ahrens, Kapulkin and Shulman.
\end{abstract}

\maketitle

In this note we remark on the problem of equality of objects in categories 
formalized in Martin-L\"of's constructive type theory.  A standard notion of category in this system
is E-category, where no such equality is specified. The main observation here is
that there is no general extension of E-categories to categories with equality on objects,  unless the principle Uniqueness of Identity Proofs (UIP) holds. In fact, for every
type $A$, there is an E-groupoid $A^{\iota}$ which cannot be so extended.
We also introduce the notion of an H-category, a variant of category, which makes it easy to 
compare to the notion of "univalent" category proposed in Univalent Type Theory \cite{HoTT}.

When formalizing mathematical structures in constructive type theory it is common
to interpret the notion of set as a type together with an equivalence relation, and
the notion of function between sets as a function or operation that preserves the
equivalence relations. Such functions are called {\em extensional functions}. This
way of interpreting sets was adopted in Bishop's seminal book \cite{B} on constructive analysis from 1967.
In type theory literature such sets are called {\em setoids}.
 Formally a setoid $X=(|X|,=_X)$ consists
of a type $|X|$ together with a binary relation $=_X$, and a proof object for $=_X$ being an equivalence relation. An extensional function between setoids $f:X \to Y$ consists of a type-theoretic function $|f| : |X| \to |Y|$, and a proof that $f$ respects the equivalence relations, i.e.\
$|f|(x) =_Y |f|(u)$ whenever $x=_X u$. One writes $x: X$ for $x:|X|$, and $f(x)$ for $|f|(x)$ to simplify notation. Every type $A$ comes with a minimal equivalence relation
 ${\rm I}_A(\cdot,\cdot)$, the so-called identity type for $A$. When the type can be inferred
 we write $a \doteq b$ for ${\rm I}_A(a,b)$. The principle of Uniqueness of Identity Proofs (UIP)  for a type $A$ states that
$$({\rm UIP}_A) \qquad \qquad (\forall a, b: A)(\forall p, q: a \doteq b)
p \doteq q.$$
This principle is not assumed in basic type theory, 
but can be proved for types $A$ where ${\rm I}_A(\cdot,\cdot)$ is a decidable relation
(Hedberg's Theorem \cite{HoTT}).

In Univalent Type Theory \cite{HoTT} the identity type
is axiomatized so as allow to quotients, and many other constructions. This makes
it possible to avoid the extra complexity of setoids and their defined equivalence relations.

These two approaches to type theory, lead to different developments of category theory. In both cases there are notions of categories, {\em E-categories} and {\em precategories}, which
are incomplete in some sense.

\section{Categories in standard type theory}

A {\em category with equality of objects} can be formulated in an essentially algebraically manner in type theory. It consists of three setoids ${\rm Ob}({\mathcal C})$, ${\rm Arr}({\mathcal C})$ and ${\rm Cmp}({\mathcal C})$ of objects, arrows and composable pairs of arrows, respectively. There are extensional functions, providing identity arrows to object,
$1: {\rm Ob} \to {\rm Arr}$,  providing domains and codomains to arrowws${\rm dom},{\rm cod}: {\rm Arr} \to {\rm Ob}$, a composition function
${\rm cmp}: {\rm Cmp} \to {\rm Arr}$,  and selection functions ${\rm fst},{\rm snd}: {\rm Cmp} \to {\rm Arr}$ satisfying familiar equations, with the convention that for a composable pair of arrows $u$:
${\rm cod}({\rm fst}(u)) = {\rm dom}({\rm snd}(u))$.
See \cite{CWM, P17} for details.

An equivalent formulation in type theory is the following \cite{P17}:
A {\em hom family presented category ${\mathcal C}$}, or just {\em HF-category}, consists
of a setoid $C$ of objects, and a (proof irrelevant) setoid family of
homomorphisms ${\rm Hom}$ indexed by the product setoid $C \times C$. Moreover
there are elements in the following dependent product setoids
\begin{itemize}
\item[(a)] $1: \Pi({\rm Ob}({\mathcal C}), {\rm Hom} \langle {\rm id}_{{\rm Ob}({\mathcal C})},{\rm id}_{{\rm Ob}({\mathcal C})} \rangle)$
\item[(b)] $\circ :\Pi({\rm Ob}({\mathcal C})^3, 
{\rm Hom}  \langle \pi_2, \pi_3 \rangle \times
{\rm Hom}  \langle \pi_1, \pi_2 \rangle \to
{\rm Hom}  \langle \pi_1, \pi_3 \rangle)$.
\end{itemize}
satisfying
\begin{itemize}
\item[] $f \circ_{a,a,b} 1_a = f  \quad 1_b \circ_{a,b,b} f = f$, if $f: {\rm Hom}(a,b)$,
\item[] $f \circ_{a,c,d} (g \circ_{a,b,c} h) = 
    (f \circ_{b,c,d} g) \circ_{a,b,d} h$, if $f: {\rm Hom}(c,d)$, $g: {\rm Hom}(b,c)$,
$h: {\rm Hom}(a,b)$.
\end{itemize}
Here $g \circ_{a,b,c} h$ is notation for the application $\circ((a,b,c), (g,h))$.

In more detail, the product setoids in (a) and (b) are made using the following constructions:

Let ${\rm Fam}(A)$ denote the type of proof irrelevant families over the setoid $A$.  Such families are closed under the following
pointwise operations:

If $F, G : {\rm Fam}(A)$, then $F \times G : {\rm Fam}(A)$ and $F \to G : {\rm Fam}(A)$.

If $F : {\rm Fam}(A)$, and $f:B \to A$ is extensional, then the composition 
$Ff : {\rm Fam}(B)$.

\medskip
The cartesian product  $\Pi(A,F)$ of a family $F: {\rm Fam}(A)$ 
consists of pairs $f= (|f|, {\rm ext}_f))$ where $f: ({\Pi} x: |A|)|F(x)|$ and ${\rm ext}_f$ is a proof that $|f|$ is extensional,
which is stated as
$$(\forall x, y:A)(\forall p: x=_A y) [|f|(F(p)(x)) =_{F(y)} |f|(y)].$$
Two such pairs $f$ and $f'$ are extensionally equally if and only if $|f|(x)=_{F(x)}|f'|(x)$
for all $x: A$. Then it is easy to check that $\Pi(A,F)$ is a setoid.

\section{E-categories and H-categories in standard type theory}

According to the philosophy  of category theory,  truly categorical notions should not refer to equality of objects. This has a very natural realization in type theory, since there, unlike in set theory, we can choose {\em not to impose} an equality on a type. This leads to the notion  of {\em E-category}.

An {\em E-category} ${\mathcal C}= (C, {\rm Hom}, \circ, 1)$ is the formulation of a 
category where there is a {\em type} $C$  of objects, but no imposed equality, and for each pair of objects $a,b$ there is a setoid ${\rm Hom}(a,b)$ of morphisms from $a$ to $b$. The composition is an extensional function
$$\circ : {\rm Hom}(b,c) \times {\rm Hom}(a,b) \to {\rm Hom}_{\mathcal C}(a,c).$$
satisfying the familiar laws of associativity and identity. 
A functor or an E-functor between E-categories is defined as usual, but the object part does not need to respect any equality of objects (because there is none). 

Now a question is whether we can impose an equality of objects onto an E-category which is compatible with composition, so as to obtain an HF-category.

Define an {\em H-category} ${\mathcal C}= (C, =_C, {\rm Hom}, \circ, 1, \tau)$ to be an E-category with an equivalence relation $=_C$ on the
objects $C$, and a family of morphisms $\tau_{a,b,p} \in {\rm Hom}(a,b)$, for each proof $p : a=_C b$. The morphisms should satisfy the conditions 
\begin{itemize}
\item[(H1)] $\tau_{a,a,p} = 1_a$ for any $p: a =_C a$
\item[(H2)] $\tau_{a,b,p} = \tau_{a,b,q}$ for any $p, q: a =_C b$
\item[(H3)] $\tau_{b,c,q} \circ \tau_{a,b,p} = \tau_{a,c,r}$ for any $p: a =_C b$, $q: b=_C c$ and
$r: a =_C c$.
\end{itemize}
Axioms (H1) and (H3) can be replaced by the special cases 
$\tau_{a,a,{\rm ref}(a)} = 1_a$, and $\tau_{b,c,q} \circ \tau_{a,b,p} = \tau_{a,c,{\rm tr}(q,p)}$ where ${\rm ref}$ and ${\rm tr}$ are specific proofs of reflexivity and transitivity.
Note that by these axioms, it follows that each $\tau_{a,b,p}$ is an isomorphism.
Specifying an H-structure on an E-category ${\mathcal C}=(C, {\rm Hom}, \circ, 1)$ then clearly amounts to
providing an equivalence relation $=_C$ and an E-functor $(C,=_C)^{\#} \to {\mathcal C}$. Here
is $(C,=_C)^{\#}$ is the E-category with objects $C$ and $(a=_Cb, \sim)$ as the  Hom setoid,
where $p \sim q$ is always true.

An H-category ${\mathcal C}$ is {\em skeletal} if $a=_Cb$ whenever $a$ and $b$
are isomorphic in ${\mathcal C}$.

\medskip
To pass between H- and HF-categories we proceed as follows:

For an H-category ${\mathcal C}= (C, =_C, {\rm Hom}, \circ, 1, \tau)$,  define a transportation
function $${\rm Hom}(p,q): 
 {\rm Hom}(a,b) \to 
{\rm Hom}(a',b')$$ for $p:a=_C a'$ and $q:b =_C b'$, by
$${\rm Hom}(p,q)(f) = \tau_{b,b',q} \circ f \circ \tau_{a',a,p^{-1}}.$$
It is straightforward to check that this defines an HF-category.

Conversely, an HF-category ${\mathcal C}= (C, {\rm Hom}, \circ, 1)$ yields an
E-category $(|C|, {\rm Hom}, \circ, 1)$ and we can define, an H-structure on
it by,  for $p: a =_C b$, 
$$\tau_{a,b,p} = {\rm Hom}(r(a),p)(1_a): {\rm Hom}(a,a) \to {\rm Hom}(a,b).$$

\medskip
 A functor between H-categories ${\mathcal C}= (C, =_C, {\rm Hom}, \circ, 1, \tau)$ and 
${\mathcal D}= (D, =_D {\rm Hom}', \circ', 1', \sigma)$ is an E-functor $F$ from 
$(C,  {\rm Hom}, \circ, 1)$ to $(D, {\rm Hom}', \circ', 1')$ such that $a =_C b$ implies
$F(a) =_D F(b)$ and $F(\tau_{a,b,p}) =\sigma_{F(a),F(b),q}$ for $p: a =_C b$ and
$q: F(a) =_D F(b)$.

\medskip
We consider the problem of extending a E-category to an H-category, first from the
classical point of view. The non-constructive
{\em Zermelo axiom of choice} (ZAC) may be stated as follows using setoids:

\begin{quote}
For any setoids  $A$ and $B$, and any binary relation $R$ between $A$ and $B$,
satisfying the totality condition $(\forall x:A)(\exists y: B)\, R(x,y)$, there is an
extensional function $f: A \to B$ such that  $(\forall x:A)\, R(x,f(x))$.
\end{quote}

The principle ZAC implies the principle of excluded middle, PEM, by Diaconescu's Theorem. Then by Hedberg's Theorem also UIP holds \cite{HoTT}.

For any type $S$, let $\hat{S}$ be the setoid $(S,{\rm I}_S(\cdot,\cdot))$. It is well known
that the special case of ZAC where $A$ is such a setoid can be proven in constructive type theory. We call this special case the {\em type theoretic axiom of choice} (TTAC).

For setoid $A$ write $\bar{A}$ for $\widehat{|A|}$.

\begin{thm} \label{selfcn} (Using ZAC) For any setoid $A$ there is an extensional function
$s:A \to \bar{A}$ such that,  for all $x, u:A$, $x =_{\bar{A}} s(x)$,
and 
\begin{equation} \label{exts}
x=_A u \Longrightarrow s(x)  =_{\bar{A}} s(u).
\end{equation}
\end{thm}
\begin{proof} Let $B= \bar{A}$. Apply ZAC to the trivally true statement
$$(\forall x:A)(\exists y: B){\rm I}_{|A|}(x,y).$$
This gives the required extensional function $s$, and extensionality implies that (\ref{exts})
holds.
\end{proof}

The significance of this theorem is that the function $s$ selects exactly one element from each
equivalence class that $=_A$ defines, i.e.\  (\ref{exts}).  Note that the existence of such
selection functions, and the TTAC implies the general ZAC. We refer to \cite{ML06} for the discussion of Zermelo's axiom of choice from the type-theoretic perspective.

\medskip
Any E-category with an equivalence relation on objects, that refine the isomorphism
relation, may be extended to an H-category using ZAC.

\begin{thm} \label{classext}
(ZAC) Assume that ${\mathcal C}= (C, {\rm Hom}, \circ, 1)$ is an E-category, 
and $=_C$ is an equivalence relation on $C$, such that $a$ and $b$ are isomorphic,
whenever $a=_C b$. Then there is a $\tau$ giving an H-structure $(=_C,\tau)$
on ${\mathcal C}$.
\end{thm}
\begin{proof} By the assumption, we have a proof object $\sigma$ such that for each $p: a=_C b$, $\sigma_{a,b,p} : {\rm Hom}(a,b)$ is an isomorphism. By Theorem \ref{selfcn} there is a 
proof object $g$ such that for all $a:C$
$$g(a): {\rm I}_{|C|}(a,s(a)).$$
Since ${\rm I}_{|C|}(\cdot,\cdot)$ is the minimal equivalence relation on $|C|$ there is a
proof object $f$ such that 
$$f: (\forall a b: |C|)({\rm I}_{|C|}(a,b) \to a=_Cb).$$
Thereby we have for each $a:C$ an isomorphism in ${\mathcal C}$,
$$\phi_a  = \sigma_{a,s(a),f(a,s(a),g(a))}: {\rm Hom}(a,s(a)).$$
Using induction on identity one defines $\rho_{a,b,p}: {\rm Hom}(a,b)$  for 
$p: {\rm I}_C(a,b)$
by $$\rho_{a,a,{\rm ref}(a)} =_{\rm def} {\rm id}_a.$$
The UIP property implies (H2).  Property (H3) follows from transitivity and (H2). 
Now by Theorem \ref{selfcn} there is a proof object $h$ such that for $a, b:C$ and
$p: a =_C b$, 
$$h(a,b,p): {\rm I}_C(s(a),s(b)).$$ 
Finally for $p:a =_C b$, we define the isomorphism
$$\tau_{a,b,p} = \phi_b^{-1} \circ \rho_{s(a),s(b),h(a,b,p)}\circ \phi_a.$$
By (H1) -- (H3) for $\rho$,  it follows, using the inverses, that also $\tau$ has these properties.
\end{proof}

\section{E-categories are proper generalizations of H-categories}

The existence of some H-structure on any E-category turns out to be equivalent to UIP.

\begin{thm} If UIP holds for the type $C$, then any E-category with objects $C$ can be extended to
an H-category.
\end{thm}
\begin{proof}
The equivalence relation on $C$ will be ${\rm I}_C(\cdot,\cdot)$.
Using induction on identity one defines $\tau_{a,b,p} \in {\rm Hom}(a,b)$  for $p \in {\rm I}(C,a,b)$
by $$\tau_{a,a,{\rm ref}(a)} =_{\rm def} {\rm id}_a.$$
The UIP property implies (H2).  Property (H3) follows from transitivity and (H2). 
\end{proof}

Let $A$ be an arbitrary type. Define the E-category $A^\iota$ where $A$ is 
the type of objects, and hom setoids are given by
$${\rm Hom}(a,b) =_{\rm def} ({\rm I}_A(a,b), \approx)$$
where $p \approx q$ holds if and only if ${\rm I}_{{\rm I}_A(a,b)}(p,q)$ is inhabited.
Let composition be given by the proof object transitivity, and the identity on $a$ is ${\rm ref}(a)$. Then it is well-known that $A^\iota$ is an E-groupoid.

\begin{thm} Let $A$ be a type. Suppose that the E-category $A^\iota$ can be extended
to an H-category. Then UIP holds for $A$.
\end{thm}
\begin{proof}  Suppose that $=_A, \tau$ is an H-structure on $A^\iota$.

Now since ${\rm I}_A(a,b)$ is the minimal equivalence relation on $A$, there is a
proof object $f(p): a=_Ab$ for each $p:{\rm I}_A(a,b)$.  Thus $\tau_{a,b,f(p)}: {\rm Hom}(a,b)= {\rm I}_A(a,b)$.  Let $D(a,b,p)$ be the proposition
\begin{equation} \label{eq1}
\tau_{a,b,f(p)} \approx p.
\end{equation}
By (H1) it holds that 
$$\tau_{a,a,f({\rm ref}(a))} \approx {\rm ref}(a),$$
i.e. $D(a,a,{\rm ref}(a))$. Hence by I-elimination (\ref{eq1}) holds. On the other hand,
(H1) gives for $p:{\rm I}_A(a,a)$, that
\begin{equation} \label{eq2}
\tau_{a,a,f(p)} \approx {\rm ref}(a).
\end{equation}
With (\ref{eq1}) this gives 
$$p \approx {\rm ref}(a)$$
for any $p:{\rm I}_A(a,a)$, which is equivalent to UIP for $A$.  
\end{proof}

\begin{cor} Assuming any E-category with $A$ as the type of objects can be extended to an H-category. Then UIP holds for $A$.
\end{cor}

In classical category theory any category maybe equipped with isomorphism as equality of objects (using Theorem \ref{classext}). This is thus {\em not} possible in basic type theory, 
with the $A^{\iota}$ as counter examples.

\section{Categories in Univalent Type Theory}  

In Univalent Type Theory \cite{HoTT}, the notion of a {\em set} is a type that satisfies the UIP condition. A {\em precategory} 
\cite[Chapter 9.1]{HoTT} is a
tuple ${\mathcal C} = (C, {\rm Hom},\circ, 1)$ where $C$ is a type, ${\rm Hom}$ is a family of types
over $C \times C$ such that ${\rm Hom}(a,b)$ is a set for any $a,b:C$. Moreover  
$1_a: {\rm Hom}(a,a)$ and
$$\circ : {\rm Hom}(b,c) \times {\rm Hom}(a,b) \to {\rm Hom}(a,c)$$
satisfy the associativity and unit laws up to ${\rm I}$-equality. 

Such a precategory thus forms an E-category by considering the hom set as the setoid 
$({\rm Hom}(a,b),{\rm I}_{{\rm Hom}(a,b)}(\cdot,\cdot))$.

Define $a \cong b$ to be the statement that $a$ and $b$ are isomorphic in ${\mathcal C}$ i.e.
$$(\exists f:{\rm Hom}(a,b))(\exists g:{\rm Hom}(b,a)) \, g\circ f \doteq 1_a \land f\circ g \doteq 1_b.$$
By ${\rm I}$-elimination one defines a function 
\begin{equation} \label{sigf}
\sigma_{a,b}: {\rm I}_C(a,b) \to a \cong b
\end{equation}
by $\sigma_{a,a}({\rm ref}(a)) = (1_a,(1_a, ({\rm ref}(1_a), {\rm ref}(1_a))))$.
Define by taking the first projection $\tau_{a,b,p} = (\sigma_{a,b}(p))_1: {\rm Hom}(a,b)$. 
By ${\rm I}$-induction 
it follows that 
\begin{itemize}
\item[] $\tau_{a,a,{\rm ref}(a)} \doteq 1_a$ for any $p: {\rm I}_C(a,a)$,
\item[] $\tau_{b,c,q} \circ \tau_{a,b,p}  \doteq \tau_{a,c,q \circ p}$ for any 
$p: {\rm I}_C(a,b)$ and $q: {\rm I}_C(b,c)$.
\end{itemize}

For a precategory where $C$ is a set, it follows that for any $p, q:{\rm I}_C(a,b)$,
${\rm I}_{{\rm I}_C(a,b)}(p,q)$ holds, so by substitution
$$\tau_{a,b,p} = \tau_{a,b,q}.$$
Thus $\tau$ gives an H-structure on $C$, so the precategory is in fact an H-category.

\medskip
An {\em univalent category} is a precategory where the function $\sigma_{a,b}$ in (\ref{sigf}) is an equivalence 
for any $a,b:C$; see \cite{AKS} and \cite[Chapter 9.1]{HoTT}. In particular, it means that if $a \cong b$, then ${\rm I}_C(a,b)$. 

An example of a precategory which is not a univalent category is given by $C={\rm N}_2$ where
${\rm Hom}(m,n) = {\rm N}_1$. Here $0 \cong 1$, but ${\rm I}_C(0,1)$ is false.

Note that an UF-category whose type of objects is a set, is a skeletal H-category.

Suppose that ${\mathcal C}$ is a skeletal precategory whose type of objects is a set. Is ${\mathcal C}$ necessarily a univalent category?
Consider the group ${\mathbb Z}_2$ as a one object, skeletal precategory: Let the underlying set be ${\rm N}_1$ and ${\rm Hom}(0,0) ={\rm N}_2$ with $0$ as unit and
$\circ$ as addition. This is not a univalent category, compare Example 9.15 in \cite{HoTT}.
Thus the standard multiplication table presentation of a nontrivial group is not a univalent category.

\end{document}